\documentclass[11pt]{article}
\usepackage{latexsym,amssymb,amsmath,amssymb}
\usepackage{amsfonts,amsbsy,amsthm}
\usepackage[all]{xy}
\numberwithin{equation}{section}
\newtheorem{theorem}{Theorem}[section]
\newtheorem{lemma}[theorem]{Lemma}
\newtheorem{proposition}[theorem]{Proposition}
\newtheorem{corollary}[theorem]{Corollary}

\theoremstyle{definition}
\newtheorem{definition}[theorem]{Definition}

\newtheorem{remark}[theorem]{Remark}
\newtheorem{example}[theorem]{Example}

\begin{document}

%%%%%%%%%%%%%%%%%%%%%%%%%%%%%%%%%%%%%%%%%%%%%%%%%%%%%%%%%%%%%%%%%%%%%

\topmargin3mm
\hoffset=-1cm
\voffset=-1.5cm
\

\medskip
\begin{center}
{\large\bf The combinatorial calculation of algebraic invariants of a monomial ideal
}
%{\large\bf The combinatorial calculation of the projective dimension, the length of Lyubeznik, and the arithmetical rank of a monomial ideal
%}
\vspace{6mm}\\
\footnotetext[1]{{\it Key words and phrases\/}.
Lyubeznik resolution, monomial ideal, Lyubeznik ideal.
\\
AMS Mathematics Subject Classification:  13A15,05E40,05E45}
\footnotetext[2]{This work was partially supported by CONACYT and PRODEP.}

\end{center}

\medskip
\begin{center}
Luis A. Dupont, Daniel G. Mendoza and Miriam Rodr\'{\i}guez
\\
{\small Facultad de Matem\'aticas, Universidad Veracruzana}\vspace{-1mm}\\
{\small Circuito Gonzalo Aguirre Beltr\'an S/N;}\vspace{-1mm}\\
{\small Zona Universitaria;}\vspace{-1mm}\\
{\small Xalapa, Ver., M\'exico, CP 91090.}\vspace{-1mm}\\
{\small e-mail: {\tt ldupont@uv.mx, dgmendozaramirez@gmail.com, \\
rodriguezolivarez19@gmail.com}\vspace{4mm}}
\end{center}

\medskip

\begin{abstract}
\noindent
We introduce the combinatorial Lyubeznik resolution of monomial ideals. We prove that this resolution is isomorphic
to the usual Lyubezbnik resolution. As an application, we give a combinatorial method to determine if an ideal
is a Lyubeznik ideal. Furthermore, the minimality of the Lyubeznik resolution is characterized and we classify all the Lyubeznik symbols using combinatorial criteria. We get a combinatorial expression for the projective dimension, the length of Lyubeznik, and the arithmetical rank of a monomial ideal. We define the Lyubeznik totally ideals as those ideals that yield a minimal free resolution under any total order. Finally, we present that for a family of graphics, that their edge ideals are Lyubeznik totally ideals.
\end{abstract}

\medskip

\section{Introduction}\label{intro}

Let $I$ be an ideal of the polynomial ring $R=K[x_1,\ldots,x_n].$
Constructing a explicit free minimal resolution of $I$ is a basic problem of combinatorial commutative algebra. In 1988, Lyubeznik \cite{Lyubeznik-resolution} constructed a graded free resolution of $R/I$ as a subcomplex
of the Taylor resolution of $R/I$. This complex is called a Lyubeznik resolution.\\

In this paper we present  the combinatorial Lyubeznik free resolution of monomial ideals which have a simplicial complex as support. This is called a Lyubeznik complex. The approach for this study is entirely combinatorial.\\

Moreover, we prove that this resolution is equal to the usual Lyu\-beznik resolution.
We see that this has two main applications. One is an extensive study of the Lyubeznik resolution from the point of view of combinatorics. The other is to set very easy combinatorial conditions to Lyubeznik resolution that tells us when the resolution is  free and  minimal. \\

The ideas for this work go back  to the papers of Novik \cite{Novik} and Guo--Wu--Yu \cite{Guo-Wu-Yu}, but our approach is different from them. \\

We develop a detailed combinatorial study over the generating set $G(I).$ We endowed it with a total order that satisfies a simple property, which guarantees a minimal free resolution. Thus, we can read the Betti numbers of $I$  from the set $G(I)$ with the considered order. In Section 2, we summarize the relevant material on Lyubeznik resolutions, and we review several facts and definitions. In Section 3, we give a detailed exposition of the combinatorial Lyubeznik resolution, which is a theory created by Guo--Wu--Yu in \cite{Guo-Wu-Yu}. In particular, we established that these resolutions are isomorphic. The  main results are in Section 4, where we classify the admissible symbols and inadmissible symbols of the algebraic Lyubeznik resolution in a combinatorial manner, which will allows us to compute the  algebraic invariants by computing combinatorial invariants that are less complicated.
\\
Finally, we present a family of graphs, whose edge ideals are Lyubeznik ideals under any total order.

\section{Lyubeznik Resolution}
Let $I=\langle\mu_{1},\mu_{2},\ldots,\mu_{f}\rangle$ be a monomial ideal in $R=K[x_1,\ldots,x_n]$  and
$$0\rightarrow \bigoplus_{j}R(-j)^{\beta_{p\,j}}
\rightarrow\cdots\rightarrow \bigoplus_{j}R(-j)^{\beta_{1\,j}}\rightarrow \bigoplus_{j}R(-j)^{\beta_{0\,j}}\rightarrow I\rightarrow 0$$
a graded minimal free resolution of $I$ over $R$. Here, $p$ is called the \emph{projective dimension} of $I$ over $R$ and denote it by $projdim(I)$. Therefore
$$0\rightarrow \bigoplus_{j}R(-j)^{\beta_{p\,j}}
\rightarrow\cdots\rightarrow \bigoplus_{j}R(-j)^{\beta_{1\,j}}\rightarrow \bigoplus_{j}R(-j)^{\beta_{0\,j}}\rightarrow R \rightarrow R/I\rightarrow 0$$
is a graded minimal free resolution of $R/I$ over $R$. Put $\beta_{i}=\sum_{j}\beta_{i\,j}$. We have $projdim(R/I)=projdim(I)+1$, $\beta_{i\,j}(I)=\beta_{(i+1)\,j}(R/I)$, $\beta_{i}(I)=\beta_{(i+1)}(R/I)$, $\mu(I)=\beta_{0}(I)$ and $\beta_{0\,j}(I)=|\{\mu_{i} : \emph{deg}(\mu_{i})=j\}|$.
\\
In general, $$\beta_{i\,j}(I)=dim(Tor_{i}^{R}(I,K))_{j}.$$
\\

Let $m_{1},m_{2},\ldots,m_{\mu}$ be an ordered sequence of $\mu$ monomials of $R$, let $I$ be the ideal
generated by these monomials.

\begin{definition}\label{definition-admisible}
For all sequences $(i_{1};i_{2};\ldots;i_{t})$, where $1\leq i_{1} < i_{2} <\cdots< i_{t}\leq\mu$, the symbol
 $u(i_{1};i_{2};\ldots;i_{t})$ will be called $L$-admissible of dimension $t$ if:
\begin{center}
$m_{q}$ does not divide
$lcm(m_{i_{h}},m_{i_{h+1}},\ldots,m_{i_{t}})$
\end{center}
for all $h < t$ and $q < i_{h}$.
We called a symbol $L$-inadmissible if is not $L$-admissible.
%  The symbols that are not $L$-admissible we define as $L$-inadmissible.
\end{definition}

Set $L^{0} = R$ and for all $t = 1,2,\ldots,\mu$ let $L^{t}$ be the free $R$-module generated by all
$L$-admissible symbols of dimension $t$. Define the map $\partial_{t}:L^{t}\rightarrow L^{t-1}$ by setting
\begin{equation*}
\partial_{t}(u(i_{1};i_{2};\ldots;i_{t}))=
\end{equation*}
\begin{equation*}
\sum_{j=1}^{t}(-1)^{j+1}\frac{lcm(m_{i_{1}},m_{i_{2}},\ldots,m_{i_{t}})}{lcm(m_{i_{1}},m_{i_{2}},\ldots \widehat{m_{i_{j}}},\ldots,m_{i_{t}})}u(i_{1};i_{2};\ldots;\widehat{i_{j}};\ldots;i_{t}).
\end{equation*}

The \emph{Lyubeznik resolution} of $R/I$, (or Lyubeznik resolution of $I$), is
a subcomplex of the Taylor resolution of $R/I$ generated by all $L$-admissible symbols.

\begin{theorem}\cite{Lyubeznik-local-coho} The complex
\begin{equation*}
\mathcal{L}\bullet_{(I,<)}:   0\rightarrow L^{\mu}\stackrel{\partial_{\mu}}{\rightarrow} L^{\mu-1}\stackrel{\partial_{\mu-1}}{\rightarrow} \cdots \stackrel{\partial_{2}}{\rightarrow} L^{1}\stackrel{\partial_{1}}{\rightarrow} L^{0}\stackrel{\partial_{0}}{\rightarrow} R/I \rightarrow0
\end{equation*}
is a free resolution of $R/I$.
\end{theorem}

Note that the Lyubeznik resolution of $I$ strictly depends on the order of the sequence $m_{1}< m_{2}< \cdots < m_{\mu}$,
different permutations of the $m_{i}$ can give rise to different resolutions. In general, the Taylor resolution of $I$ is far from being a minimal graded free
resolution but a Lyubeznik resolution of $I$ often gives a minimal graded free resolution or a graded free resolution whose length is equal to the projective
dimension of $R/I$.
\\

For two $L$-admissible symbols $u(i_{1};i_{2};\ldots;i_{s})$ and $u(j_{1};j_{2};\ldots;j_{t})$, we say that
\begin{equation*}
  u(i_{1};i_{2};\ldots;i_{s}) \preceq u(j_{1};j_{2};\ldots;j_{t})
\end{equation*}
if $i_{1},i_{2},\ldots,i_{s}$ is a subsequence of $j_{1},j_{2},\ldots,i_{t}$. Evidently, if $u(j_{1};j_{2};\ldots;j_{t})$ is
$L$-admissible, so are all the smaller symbols. Hence every Lyubeznik resolution is uniquely
determined by its maximal $L$-admissible symbols.

\begin{definition}\label{symbol-broken}
A symbol  $u(i_{1};i_{2};\ldots;i_{t})$ is \emph{stable} of $I$, if for all $1\leq q\leq t$
\begin{equation*}
lcm(m_{i_{1}},m_{i_{2}},\ldots,m_{i_{t}})\neq lcm(m_{i_{1}},m_{i_{2}},\ldots \widehat{m_{i_{q}}},\ldots,m_{i_{t}}).
\end{equation*}
Note that if $u(i_{1};i_{2};\ldots;i_{t})$ is stable, then also are all
smaller $L$-admissible symbols.
\end{definition}

Let $\verb"m"$ be the homogeneous maximal ideal of $R$, i.e., $\verb"m" = \langle x_{1}, x_{2},\ldots, x_{n}\rangle$. $L_{\bullet}$ is minimal if and only
if $\partial_{t}(L^{t})\subseteq \verb"m"L^{t-1}$ for all $t$. By the construction of $\partial_{t}$, $L_{\bullet}$ is minimal if and only if
for all maximal $L$-admissible symbols $u(i_{1};i_{2};\ldots;i_{t})$, $u$ is stable.
\\
We have the following proposition:
\begin{proposition}\label{betti-simbolos}
The Lyubeznik resolution of I with respect to some order of monomial
ge\-ne\-ra\-tors is the minimal free resolution if and only if all maximal $L$-admissible symbols $u(i_{1};i_{2};\ldots;i_{t})$ are stable. In particular,
%\begin{equation*}
%\beta_{t\, j}(R/I)=\beta_{t-1\, j}(I)=|\{u(i_{1};i_{2};\ldots;i_{t}) \,:\, u\emph{ is L-admissible}; j=deg(lcm(m_{i_{1}},m_{i_{2}},\ldots,m_{i_{t}}))\}|.
%\end{equation*}
\[\beta_{t\, j}(R/I)=\beta_{t-1\, j}(I)  \textmd{ and } \]
\[\beta_{t-1\, j}(I)= |\{u(i_{1};i_{2};\ldots;i_{t}) \,:\, u \textmd{ is L-admissible}; j=deg(lcm(m_{i_{1}},m_{i_{2}},\ldots,m_{i_{t}}))\}|.\]
\end{proposition}

We define the \emph{L--length of Lyubeznik (With respect to $\prec$)}, $L_{\prec}(I)$, as the length of the Lyubeznik resolution of $R/I$ with respect to the order $\prec$. Also defined the \emph{L--length of Lyubeznik} by:
$$L(I)=\textmd{min}\{L_{\prec}(I) : (G(I),\prec) \textmd{ a monomial order } \}.$$So, $$projdim(R/I)\leq L(I).$$
\begin{definition}
For a monomial ideal $I$, let $G(I)$ be its minimal set of monomial generators. If there is a
total order on $G(I)$ such that the corresponding Lyubeznik resolution of $R/I$ is a minimal free resolution of
$R/I$, then $I$ is called a \emph{Lyubeznik ideal}. In this case, $L(I)=projdim(R/I)$. In addition, we have defined that an ideal $I$ is called an \emph{almost Lyubeznik ideal} when $L(I)=projdim(R/I)$, (But not necessarily the corresponding Lyubeznik resolution of $R/I$ is a minimal free resolution of
$R/I$).
\end{definition}

\section{Meaning of the Lyubeznik resolutions from the point of view of the combinatorial}

Let $I=\langle\mu_{1},\mu_{2},\ldots,\mu_{f}\rangle$ be a monomial ideal in $R=K[x_1,\ldots,x_n]$, let $G(I)=\{\mu_{1},\mu_{2},\ldots,\mu_{f}\}$ be its minimal set of monomial
generators.
\\

The following definitions are necessary for our study of the resolution of Lyubeznik and do not depend on the monomial order considered on $G(I)$.
\begin{definition}
For a subset $A \subseteq G(I)$, the \emph{multidegree} of $A$, denoted by $lcm(A)$, is the least
common multiple of the elements in $A$. We call a subset $C$ of $G(I)$ a \emph{cover} of a monomial
$u \in C$, if $u | lcm(C-\{u\})$, or alternatively we say $C$ covers $u$, denoted by $u\lhd C$. The
\emph{complete cover} induced by a cover $C$, denoted by $\overline{C}$, is $\overline{C}=\{w\in G(I) : w | lcm(C)\}$. A cover $C$ (of $u$) is called an $M$-\emph{minimal
cover} of $G(I)$, if there exists no cover $V$ (of some $v$) whose multidegree $lcm(V)$ is a proper
factor of $lcm(C)$. A cover $C$ of a monomial $u$ is called an $E$-\emph{minimal cover} of $u$ if no proper
subset of $C$ can cover $u$.
\end{definition}

Note that the following definitions depend on the monomial order considered.
\\

Let $\prec$ be a total order on $G(I)$, and let $A$ be a subset of $G(I)$. Let $min(A)$ be
the least element of $A$ under the total order $\prec$. Let $B$ be another subset of $G(I)$. If
$min(A) \prec min(B)$, then we write $A \prec B$. If $A$ has only one element $u$ and $u \prec min(B)$,
then we denote $u \prec B$. A set $D$ is said to be \emph{broken} under the total order $\prec$, if there
exists an element $u \in G(I)$, such that $u | lcm(D)$ and $u \prec D$. In these conditions, we say that $u$ is a \emph{court} of $D$. A subset $E$ of $G(I)$ is called
\emph{preserved}, if no subset of $E$ is broken.
Let $\triangle_{I}$ be the full simplex on $G(I)$. Let $\Delta_{(I,\prec)}$ be
the following simplicial subcomplex of $\triangle_{I}$:
$$\Delta_{(I,\prec)}=\{F\in \triangle_{I} : min\{u\in G(I): u|lcm(H)\}\in H \textmd{  for all  } H\subseteq F\}.$$
The following associated algebraic chain complex, $\mathcal{CL}\bullet_{(I,<)}$, is proved to be a free resolution of $I$, and is called the \emph{combinatorial Lyubeznik resolution}
of $I$ under the total order $\prec$ and $\Delta_{(I,\prec)}$ is the \emph{Lyubeznik complex}
of $I$ under the total order $\prec$.
\begin{equation*}
\mathcal{CL}\bullet_{(I,<)}:   \cdots\rightarrow L^{n}\stackrel{\partial_{n}}{\rightarrow} L^{n-1}\stackrel{\partial_{n-1}}{\rightarrow} \cdots \stackrel{\partial_{2}}{\rightarrow} L^{1}\stackrel{\partial_{1}}{\rightarrow} L^{0}\stackrel{\partial_{0}}{\rightarrow} R/I \rightarrow0
\end{equation*}
Where, $\Delta_{i}=\{F\in \Delta_{(I,\prec)} : |F|=i\}$. For a given $F=\{u_{j 1},u_{j 2},\ldots,u_{j i}\}\in \Delta_{i}$, let $G_{k}=F-\{u_{j k}\}\in L_{i-1}$, $1\leq k\leq i$. With
$$\partial(F)=\sum_{k=1}^{i}\varepsilon_{F}^{G_{k}}\frac{lcm(F)}{lcm(G_{k})},$$ where the sign $\varepsilon_{F}^{G_{k}}$ equals to $1$ (respectively, $-1$) for odd $k$ (for even $k$, respectively).
\\

The next Theorem gives us the equivalence between the Lyubeznik resolution and the combinatorial Lyubeznik resolution.
\begin{theorem}\label{carasimplicial=simboloadmisible}
Let $(G(I),<)$ be a monomial order, and let $F$ be a subset of $G(I)$. Then $F=\{m_{i_{1}}< m_{i_{2}}<\cdots< m_{i_{t}}\} \in \Delta_{(I,\prec)}$ if and only if $u(m_{i_{1}};m_{i_{2}};\ldots;m_{i_{t}})$ is a $L$-admissible symbol. In particular, $F=\{m_{i_{1}}< m_{i_{2}}<\cdots< m_{i_{t}}\}$ is a facet of $\Delta_{(I,\prec)}$ if and only if $u(m_{i_{1}};m_{i_{2}};\ldots;m_{i_{t}})$ is a maximal $L$-admissible symbol.
\end{theorem}
\begin{proof}
Suppose that $u(m_{i_{1}};m_{i_{2}};\ldots;m_{i_{t}})$ is an $L$-inadmissible symbol. There exists $h<t;q<i_{h}$, where $m_{i_{h}}m_{i_{h+1}}\cdots m_{i_{t}}$ is divisible by $m_{q}$.
Therefore $m_{q} < A:=\{m_{i_{h}}< m_{i_{h+1}}<\cdots< m_{i_{t}}\}$ and $m_{q}$ divides $lcm(A)$. Hence $min\{u\in G(I): u|lcm(A)\}\leq m_{q} < A$ implies that
$min\{u\in G(I): u|lcm(A)\} \notin A\subseteq F$ with $F\in \Delta_{(I,\prec)}$, which is a contradiction.
\\
Now, suppose that $u(m_{i_{1}},m_{i_{2}},\ldots,m_{i_{t}})$ is an $L$-admissible symbol and $F=\{m_{i_{1}}< m_{i_{2}}<\cdots< m_{i_{t}}\} \notin \Delta_{(I,\prec)}$. There exists $G\subseteq F$ such that $m_{q}=min\{u\in G(I): u|lcm(G)\}\notin G$, in particular $|G|>1$. Consider $G=\{m_{j_{1}}< m_{j_{2}}<\cdots< m_{j_{s}}\}\subseteq F$ and $(j_{1},j_{2},\ldots,j_{s})$ a subsequence of $(i_{1},i_{2},\ldots,i_{t})$ with $s\geq2$. Note that $u(i_{1};i_{2};\ldots;i_{t})$ $L$-admissible implies that $u(j_{1};j_{2};\ldots;j_{s})$ is $L$-admissible. Furthermore, $1<s,q<j_{1}$ and $$m_{q}|lcm(m_{j_{1}},m_{j_{2}},\ldots,m_{j_{s}}).$$
This contradicts to the fact that $u(j_{1};j_{2};\ldots;j_{s})$ is $L$-admissible.
\end{proof}

\begin{corollary}\label{lyubeznik=Lyubeznikcombinat}
For a given monomial order, $(G(I),<)$, $\mathcal{L}\bullet_{(I,<)}=\mathcal{CL}\bullet_{(I,<)}$. That is to say, the Lyubeznik resolution and the combinatorial Lyubeznik resolution are equal.
\end{corollary}

We recall the definition of preserved set.
\begin{definition}
Let $(G(I),<)$ be a monomial order, and let $F=\{m_{i_{1}}< m_{i_{2}}<\cdots< m_{i_{t}}\}$ be a subset of $G(I)$. Then $F$ is an \emph{preserved set} if
$$\forall A\subseteq F, \textmd{min}\{u\in G(I): u|m(A)\}\in A, \textmd{ equivalently}, F\in \Delta_{(I,\prec)}.$$
\end{definition}
\begin{remark}
Let $(G(I),<)$ be a monomial order, and let $F=\{m_{i_{1}}< m_{i_{2}}<\cdots< m_{i_{t}}\}$ be a preserved subset of $G(I)$, then $F$ is not broken. The  reciprocal is false as shown in the following example:
\end{remark}
\begin{example}
Consider the following monomial ideal
$$I=\langle\underbrace{x^{2}y}_{m_{1}},\underbrace{y^{2}z}_{m_{2}},\underbrace{x^{3}}_{m_{3}}, \underbrace{y^{3}}_{m_{4}}, \underbrace{z^{3}}_{m_{5}}\rangle,$$
and the monomial order in $G(I)=\{m_{1}<m_{2}<m_{3}<m_{4}<m_{5}\}$. Where $\{m_{1}, m_{2}, m_{3}\}$ is not broken, but neither is preserved since the subset $\{m_{2},m_{3}\}$ is broken, with the court $m_{1}$.
\end{example}
We have in combinatorial, an equivalent definition of an $L$-admissible symbol.
\begin{corollary}
Let $(G(I),<)$ be a monomial order, and let $F=\{m_{i_{1}}< m_{i_{2}}<\cdots< m_{i_{t}}\}$ be a subset of $G(I)$. Then $u(m_{i_{1}};m_{i_{2}};\ldots;m_{i_{t}})$ is an $L$-admissible symbol, if and only if, $F$ is a preserved set.
\end{corollary}

\begin{lemma}\label{lema1}
Let $(G(I),<)$ be a monomial order, and let $C=\{m_{i_{1}}< m_{i_{2}}<\cdots< m_{i_{t}}\}$ be a subset of $G(I)$. If $C$ is a cover and $u(m_{i_{1}};m_{i_{2}};\ldots;m_{i_{t}})$ is an $L$-admissible symbol, then $u(m_{i_{1}};m_{i_{2}};\ldots;m_{i_{t}})$ is an $L$-admissible symbol which is not stable.
\end{lemma}
\begin{proof}
Suppose that $v\lhd C$ for some $m_{i_{k}}=v$. This implies that $$lcm(m_{i_{1}},\ldots,\widehat{m_{i_{k}}},\ldots,m_{i_{t}})=lcm(m_{i_{1}},m_{i_{2}},\ldots,m_{i_{t}}).$$ Hence, $u(m_{i_{1}};m_{i_{2}};\ldots;m_{i_{t}})$ is an $L$-admissible symbol which is not stable.
\end{proof}
\begin{lemma}\label{lema2}
Let $(G(I),<)$ be a monomial order, and let $C=\{m_{i_{1}}< m_{i_{2}}<\cdots< m_{i_{t}}\}$ be a subset of $G(I)$. If $C$ is a preserved cover, then $u(m_{i_{1}};m_{i_{2}};\ldots;m_{i_{t}})$ is an $L$-admissible symbol which is not stable.
\end{lemma}
\begin{proof}
Suppose that $C=\{m_{i_{1}}< m_{i_{2}}<\cdots< m_{i_{t}}\}$ is a preserved cover and $u(m_{i_{1}};m_{i_{2}};\ldots;m_{i_{t}})$ is an $L$-inadmissible symbol. There exists $h<t;q<i_{h}$, where $m_{i_{h}}m_{i_{h+1}}\cdots m_{i_{t}}$ is divisible by $m_{q}$.
Therefore $m_{q} < A:=\{m_{i_{h}}< m_{i_{h+1}}<\cdots< m_{i_{t}}\}$ and $m_{q}$ divides $lcm(A)$. So, $A$ is broken, implies that
$C$  is not preserved, which is a contradiction. Hence, $u(m_{i_{1}};m_{i_{2}};\ldots;m_{i_{t}})$ is an $L$-admissible symbol. By Lemma~\ref{lema1}, $u(m_{i_{1}};m_{i_{2}};\ldots;m_{i_{t}})$ is a not stable $L$-admissible symbol.
\end{proof}
\begin{lemma}\label{lema3}
Let $(G(I),<)$ be a monomial order, and let $C=\{m_{i_{1}}< m_{i_{2}}<\cdots< m_{i_{t}}\}$ be a subset of $G(I)$. If $u(m_{i_{1}};m_{i_{2}};\ldots;m_{i_{t}})$ is a not stable $L$-admissible symbol, then $C$ is a preserved cover.
\end{lemma}
\begin{proof}
Suppose that $u(m_{i_{1}};m_{i_{2}};\ldots;m_{i_{t}})$ is an $L$-admissible symbol that is not stable.  There exists $m_{i_{k}}\in C$ such that $lcm(m_{i_{1}},\ldots,\widehat{m_{i_{k}}},\ldots,m_{i_{t}})=lcm(m_{i_{1}},m_{i_{2}},\ldots,m_{i_{t}})$. Therefore, $m_{i_{k}}\lhd C$. By the definition of $\Delta_{(I,\prec)}$, $C$ is preserved. Otherwise, there exists $A\subseteq C$ and $m_{q} \in G(I)-A$ such that  $m_{q} < A$ and $m_{q}$ divides $lcm(A)$. This contradicts to the fact that $A\in \Delta_{(I,\prec)}$.
\end{proof}
\begin{proposition}\label{s.a.noestable=cubiertapreservada}
Let $(G(I),<)$ be a monomial order, and let $C=\{m_{i_{1}}< m_{i_{2}}<\cdots< m_{i_{t}}\}$ be a subset of $G(I)$. Then $C$ is a preserved cover, if and only if, $u(m_{i_{1}};m_{i_{2}};\ldots;m_{i_{t}})$ is a not stable $L$-admissible symbol.
\end{proposition}
\begin{proof}
Lemmas~(\ref{lema2}) and ~(\ref{lema3}).
\end{proof}
\begin{proposition}\label{equiva1}
Let $(G(I),<)$ be a monomial order. Then the following statements are equivalent
\begin{description}
  \item[(i)] For every element $u$ of $G(I)$ and
every cover $C$ of $u$, $C$ is not preserved.
  \item[(ii)] For every element $u$ of $G(I)$ and
every cover $C$ of $u$, there exist $D \subseteq C$ and $v \notin D$, such that $D\cup\{v\}$ is a
cover of $v$, satisfying $\textmd{min}(\overline{D}-D)< \textmd{min}(D)$.
\end{description}
\end{proposition}
\begin{proof}
The condition of that $C$ is not preserved implies that there exists a broken subset $D$ of $C$,
such that some element $v\notin D$ with $v < D$ holds that $v|lcm(D)$, implies that $v \lhd D\cup\{v\}$, satisfying $\textmd{min}(\overline{D}-D)\leq v < \textmd{min}(D)$.
Conversely, suppose $(ii)$. Let $u\lhd C$ be a cover of $u$. Then there exist $D \subseteq C$ and $v \notin D$, such that $D\cup\{v\}$ is a
cover of $v$, satisfying $\textmd{min}(\overline{D}-D)< \textmd{min}(D)$. Hence $D$ is broken and $C$ is not preserved.
\end{proof}

With a similar proof, and considering a minimal cover in each cover, we get the following result.
\begin{proposition}\label{equiva2}
Let $(G(I),<)$ be a monomial order. Then the following statements are equivalent
\begin{description}
  \item[(i)] For every element $u$ of $G(I)$ and
every $E$--minimal cover $C$ of $u$, $C$ is not preserved.
  \item[(ii)] For every element $u$ of $G(I)$ and
every $E$--minimal cover $C$ of $u$, there exist $D \subseteq C$ and $v \notin D$, such that $D\cup\{v\}$ is an $E$--minimal
cover of $v$, satisfying $\textmd{min}(\overline{D}-D)< \textmd{min}(D)$.
\end{description}
\end{proposition}

We can summarize the above results in the following theorem
\begin{theorem}\label{equivalenciasdeideallyubeznik}
Let $(G(I),<)$ be a monomial order. Then the following statements are equivalent
\begin{description}
  \item[(i)] The Lyubeznik resolution, $\mathcal{L}\bullet_{(I,<)}$, of $I$ with respect to the monomial order $(G(I),<)$ is a minimal free resolution.
  \item[(ii)] For every element $u$ of $G(I)$ and
every cover $C$ of $u$, $C$ is not preserved.
  \item[(iii)] For every element $u$ of $G(I)$ and
every cover $C$ of $u$, there exist $D \subseteq C$ and $v \notin D$, such that $D\cup\{v\}$ is a
cover of $v$, satisfying $\textmd{min}(\overline{D}-D)< \textmd{min}(D)$.
  \item[(iv)] For every element $u$ of $G(I)$ and
every $E$--minimal cover $C$ of $u$, $C$ is not preserved.
  \item[(v)] For every element $u$ of $G(I)$ and
every $E$--minimal cover $C$ of $u$, there exist $D \subseteq C$ and $v \notin D$, such that $D\cup\{v\}$ is an $E$--minimal
cover of $v$, satisfying $\textmd{min}(\overline{D}-D)< \textmd{min}(D)$.
\end{description}
\end{theorem}
\begin{proof}
This follows from Propositions~(\ref{betti-simbolos}),~(\ref{s.a.noestable=cubiertapreservada}),~(\ref{equiva1}) and ~(\ref{equiva2}).
\end{proof}
\begin{corollary}
Let $I$ be a monomial ideal in the polynomial ring $R=K[x_{1},x_{2},\ldots,x_{n}]$. Then the following statements are equivalent
\begin{description}
  \item[(i)] $I$ is a Lyubeznik ideal.
  \item[(ii)] There exists a monomial order $(G(I),<)$, such that for every element $u$ of $G(I)$ and
every cover $C$ of $u$, $C$ is not preserved.
  \item[(iii)] There exists a monomial order $(G(I),<)$, such that for every element $u$ of $G(I)$ and
every cover $C$ of $u$, there exist $D \subseteq C$ and $v \notin D$, such that $D\cup\{v\}$ is a
cover of $v$, satisfying $\textmd{min}(\overline{D}-D)< \textmd{min}(D)$.
  \item[(iv)] There exists a monomial order $(G(I),<)$, such that for every element $u$ of $G(I)$ and
every $E$--minimal cover $C$ of $u$, $C$ is not preserved.
  \item[(v)] There exists a monomial order $(G(I),<)$, such that for every element $u$ of $G(I)$ and
every $E$--minimal cover $C$ of $u$, there exist $D \subseteq C$ and $v \notin D$, such that $D\cup\{v\}$ is an $E$--minimal
cover of $v$, satisfying $\textmd{min}(\overline{D}-D)< \textmd{min}(D)$.
\end{description}
\end{corollary}

\begin{example}
Consider the following monomial ideal
$$I=\langle\underbrace{x^{2}y}_{m_{1}},\underbrace{y^{2}z}_{m_{2}},\underbrace{x^{3}}_{m_{3}}, \underbrace{y^{3}}_{m_{4}}, \underbrace{z^{3}}_{m_{5}}\rangle,$$
and the monomial order in $G(I)=\{m_{1}<m_{2}<m_{3}<m_{4}<m_{5}\}$.
The covers of $m_{1}$ are
\begin{itemize}
\item $C_{1}=\{m_{1}, m_{2}, m_{3}, m_{4}, m_{5}\}$,
\item $C_{2}=\{m_{1}, m_{3}, m_{4}, m_{5}\}$,
\item $C_{3}=\{m_{1}, m_{2}, m_{3}, m_{5}\}$,
\item $C_{4}=\{m_{1}, m_{2}, m_{3}\}$ ($E$-minimal),
\item $C_{5}=\{m_{1}, m_{3}, m_{4}\}$ ($E$-minimal),
\item $C_{6}=\{m_{1}, m_{2}, m_{3}, m_{4}\}$.
\end{itemize}
The covers of $m_{2}$ are
\begin{itemize}
\item $C_{7}=\{m_{2}, m_{4}, m_{5}\}$ ($E$-minimal),
\item $C_{8}=\{m_{1}, m_{2}, m_{4}, m_{5}\}$,
\item $C_{9}=\{m_{2}, m_{3}, m_{4}, m_{5}\}$,
\item $C_{1}$.
\end{itemize}
Note that, $m_{3}$, $m_{4}$ and $m_{5}$  have no covers. Furthermore, all covers are not preserved
\begin{itemize}
  \item $C_{1^{*}}=\{m_{2}, m_{3}, m_{4}, m_{5}\}$ is a broken subset of $C_{1}$ with court $m_{1}$.
  \item $C_{2^{*}}=\{ m_{3}, m_{4}, m_{5}\}$ is a broken subset of $C_{2}$ with court $m_{1}$.
  \item $C_{3^{*}}=\{m_{2}, m_{3}, m_{5}\}$ is a broken subset of $C_{3}$ with court $m_{1}$.
  \item $C_{4^{*}}=\{m_{2}, m_{3} \}$ is a broken subset of $C_{4}$ with court $m_{1}$.
  \item $C_{5^{*}}=\{m_{3}, m_{4}\}$ is a broken subset of $C_{5}$ with court $m_{1}$.
  \item $C_{6^{*}}=\{m_{3}, m_{4}\}$ is a broken subset of $C_{6}$ with court $m_{1}$.
  \item $C_{7^{*}}=\{m_{4}, m_{5}\}$ is a broken subset of $C_{7}$ with court $m_{2}$.
  \item $C_{8^{*}}=\{m_{4}, m_{5}\}$ is a broken subset of $C_{8}$ with court $m_{2}$.
  \item $C_{9^{*}}=\{m_{4}, m_{5}\}$ is a broken subset of $C_{9}$ with court $m_{2}$.
\end{itemize}
\begin{remark}
Por medio de comprobación directa calculando las cubiertas en cada orden posible obtenemos lo siguiente:
$I=\langle x^{2}y^{2}< t^{2}z^{2}< x^{2}z^{2}\rangle$ $\subseteq$ $J=\langle x^{2}y^{2}, t^{2}z^{2}, x^{2}z^{2}, t^{2}y^{2}\rangle$
$\subseteq$ $K=\langle x^{2}y^{2}< t^{2}z^{2}< x^{2}z^{2}< t^{2}y^{2}< xyzt\rangle$, donde $I,K$ son ideales de Lyubeznik y $J$ no lo es.
\end{remark}

 As well, by the Theorem~\ref{equivalenciasdeideallyubeznik} the Lyubeznik resolution of $I$ with respect to the monomial order given in $G(I)$ is a minimal free resolution. Therefore, the ideal $I$ is a Lyubeznik ideal.
\end{example}

\section{Classification of the symbols of $I$ and the subsets of $G(I)$}

\begin{proposition}\label{s.estable=nocubierta}
Let $(G(I),<)$ be a monomial order, and let $F=\{m_{i_{1}}< m_{i_{2}}<\cdots< m_{i_{t}}\}$ be a subset of $G(I)$. Then $u(m_{i_{1}};m_{i_{2}};\ldots;m_{i_{t}})$ is a stable symbol, if and only if, $F$ is not cover.
\end{proposition}
\begin{proof}
Suppose that $u(m_{i_{1}};m_{i_{2}};\ldots;m_{i_{t}})$ is a stable symbol. Then $m_{i_{k}}$ does not divide
$$lcm(m_{i_{1}},\ldots,\widehat{m_{i_{k}}},\ldots,m_{i_{t}})$$ for all $1\leq k\leq t$. Hence $F$ is not cover.
Conversely, if $F$ is not a cover, then $$lcm(m_{i_{1}},\ldots,\widehat{m_{i_{k}}},\ldots,m_{i_{t}})\neq lcm(m_{i_{1}},m_{i_{2}},\ldots,m_{i_{t}}).$$ Hence, $u(m_{i_{1}};m_{i_{2}};\ldots;m_{i_{t}})$ is a stable symbol.
\end{proof}
\begin{proposition}\label{s.noadmisible=nopreservado}
Let $(G(I),<)$ be a monomial order, and let $F=\{m_{i_{1}}< m_{i_{2}}<\cdots< m_{i_{t}}\}$ be a subset of $G(I)$. Then $u(m_{i_{1}};m_{i_{2}};\ldots;m_{i_{t}})$ is a symbol that is  $L$-inadmissible, if and only if, $F$ is not preserved.
\end{proposition}
\begin{proof}
$u(m_{i_{1}};m_{i_{2}};\ldots;m_{i_{t}})$ is an $L$-inadmissible symbol $\Leftrightarrow$ $F \notin \Delta_{(I,\prec)}$ $\Leftrightarrow$
there exists $A\subseteq F$ such that $\textmd{min}\{u\in G(I) : u|lcm(A)\}\notin A$ $\Leftrightarrow$ there exists $A\subseteq F$ broken $\Leftrightarrow$ $F$ is not preserved.
\end{proof}
\begin{theorem}\label{clasificasimbolos}\emph{(Classification of the symbols of $I$)}
Let $(G(I),<)$ be a monomial order, and let $F=\{m_{i_{1}}< m_{i_{2}}<\cdots< m_{i_{t}}\}$ be a subset of $G(I)$. Then $u(m_{i_{1}};m_{i_{2}};\ldots;m_{i_{t}})$ is a symbol that satisfies only one of the following three properties:
\begin{itemize}
    \item is an $L$-admissible symbol that is not stable and $F$  is a preserved cover.
    \item is a stable $L$-admissible symbol and $F$ is preserved and is not a cover.
    \item is a stable  symbol that is a $L$-inadmissible and $F$ is not preserved set and is not cover.
    \item is a non-stable symbol that is $L$-inadmissible and $F$ is not preserved set and $F$ is a cover.
  \end{itemize}
\end{theorem}
\begin{proof}
This follows from the Propositions~(\ref{s.a.noestable=cubiertapreservada}),~(\ref{s.estable=nocubierta}) and~(\ref{s.noadmisible=nopreservado}).
\end{proof}
\begin{theorem}\label{clasificaconjuntos}\emph{(Classification of the subsets of $G(I)$)}
Let $(G(I),<)$ be a monomial order, and let $F=\{m_{i_{1}}< m_{i_{2}}<\cdots< m_{i_{t}}\}$ be a subset $G(I)$. Then $F$ satisfies only one of the following four properties:
\begin{itemize}
  \item is a preserved cover and $u(m_{i_{1}};m_{i_{2}};\ldots;m_{i_{t}})$  is a $L$-admissible symbol that is not stable.
  \item is a cover that is not preserved and  $u(m_{i_{1}};m_{i_{2}};\ldots;m_{i_{t}})$  is a symbol that is  $L$-inadmissible and is not stable.
  \item is not cover, $F$ is preserved and  $u(m_{i_{1}};m_{i_{2}};\ldots;m_{i_{t}})$  is a stable $L$-admissible symbol.
  \item is not cover, $F$ is not preserved and  $u(m_{i_{1}};m_{i_{2}};\ldots;m_{i_{t}})$ is a stable symbol that is  $L$-inadmissible.
\end{itemize}
\end{theorem}
\begin{proof}
This also follows from the Propositions~(\ref{s.a.noestable=cubiertapreservada}),~(\ref{s.estable=nocubierta}) and~(\ref{s.noadmisible=nopreservado}).
\end{proof}
\begin{example}
Consider the following monomial ideal
$$I=\langle\underbrace{x_{1}x_{2}x_{4}}_{m_{1}}, \underbrace{x_{1}x_{2}x_{3}}_{m_{2}},\underbrace{x_{1}x_{5}}_{m_{3}},\underbrace{x_{2}x_{3}x_{6}}_{m_{4}},\underbrace{x_{4}x_{6}}_{m_{5}}\rangle,$$
and the monomial order in $G(I)=\{m_{1}<m_{2}<m_{3}<m_{4}<m_{5}\}$.
The covers of $m_{1}$ are
\begin{itemize}
\item $C_{1}=\{m_{1}, m_{2}, m_{5}, m_{3}\}$,
\item $C_{2}=\{m_{1}, m_{2}, m_{5}, m_{4}\}$,
\item $C_{3}=\{m_{1}, m_{3}, m_{4}, m_{5}, m_{2}\}$,
\item $C_{4}=\{m_{1}, m_{3}, m_{4}, m_{5}\}$,
\item $C_{5}=\{m_{1}, m_{2}, m_{5}\}$.
\end{itemize}
The covers of $m_{2}$ are
\begin{itemize}
\item $C_{6}=\{m_{2}, m_{3}, m_{4}, m_{5}\}$,
\item $C_{7}=\{m_{2}, m_{3}, m_{4}, m_{1}\}$,
\item $C_{8}=\{m_{2}, m_{1}, m_{4}, m_{5}\}$,
\item $C_{9}=\{m_{2}, m_{3}, m_{4}\}$,
\item $C_{10}=\{m_{2}, m_{1}, m_{4}\}$.
\end{itemize}
Note that, $m_{3}$ have no covers. The covers of $m_{4}$ are
\begin{itemize}
\item $C_{11}=\{m_{4}, m_{2}, m_{5}, m_{3}\}$,
\item $C_{12}=\{m_{4}, m_{2}, m_{5}, m_{1}\}$,
\item $C_{13}=\{m_{4}, m_{2}, m_{5}\}$.
\end{itemize}
The covers of $m_{5}$ are
\begin{itemize}
\item $C_{14}=\{m_{5}, m_{1}, m_{4}, m_{2}\}$,
\item $C_{15}=\{m_{5}, m_{1}, m_{4}, m_{3}\}$,
\item $C_{16}=\{m_{5}, m_{1}, m_{4}\}$.
\end{itemize}
We note that
\begin{itemize}
\item[(i)] $C_{2}=C_{8}=C_{12}=C_{14}$, $C_{4}=C_{15}$ and $C_{6}=C_{11}$.
\item[(ii)] The $E$-minimal covers of $(G(I),<)$ are $C_{5}$, $C_{9}$, $C_{10}$, $C_{13}$ and $C_{16}$.
\item[(iii)] $C_{10}$ and $C_{16}$ are the only preserved covers.
\end{itemize}
The set of $L$-admissible symbols of $I$ with dimension $2$ is $$L_{2}=\{u(m_{1};m_{2}), u(m_{1};m_{3}), u(m_{1};m_{4}), u(m_{1};m_{5}), u(m_{2};m_{3}), u(m_{2};m_{4}), u(m_{3};m_{5}), u(m_{4};m_{5})\}.$$
The set of $L$-admissible symbols of $I$ with dimension $3$ is
$$L_{3}=\{u(m_{1};m_{2};m_{3}), u(m_{1};m_{2};m_{4}), u(m_{1};m_{3};m_{5}), u(m_{1};m_{4};m_{5})\}.$$
The set of $L$-inadmissible symbols of $I$ with dimension $2$ is
$$L_{2}^{'}=\{u(m_{2}; m_{5}), u(m_{3}, m_{4})\}.$$
The set of $L$-inadmissible symbols of $I$ with dimension $3$ is
$$L_{3}^{'}=\{u(m_{1}, m_{2}, m_{5}), u(m_{1}, m_{3}, m_{4}), u(m_{2}, m_{3}, m_{4}), u(m_{2}, m_{3}, m_{5}), u(m_{2}, m_{4}, m_{5}), u(m_{3}, m_{4}, m_{5})\}.$$
The set of $L$-inadmissible symbols of $I$ with dimension $4$ is
$$L_{4}^{'}=\{u(m_{1}; m_{2}; m_{5}; m_{3}), u(m_{1}; m_{2}; m_{5}; m_{4}), u(m_{1}; m_{3}; m_{4}; m_{5}), u(m_{2}; m_{3}; m_{4}; m_{5}), u(m_{2}; m_{3}; m_{4}; m_{1})\}.$$
The set of $L$-inadmissible symbols of $I$ with dimension $5$ is $L_{5}^{'}=\{u(m_{1}; m_{2}; m_{3}; m_{4};m_{5})$.
\\
We note that
\begin{itemize}
\item[(iv)] The $L$-admissible symbols $u(m_{1};m_{2};m_{4})$ and $u(m_{1};m_{4};m_{5})$, are the only non-stable $L$-admissible symbols.
\item[(iv$^{'}$)] The covers $C_{10}=\{m_{1}, m_{2}, m_{4}\}$ and $C_{16}=\{m_{1}, m_{4}, m_{5}\}$ are the only preserved covers.
\item[(v)] The symbols $u(m_{1};m_{2};m_{5})$, $u(m_{2};m_{3};m_{4})$ and $u(m_{2};m_{4};m_{5})$, are $L$-inadmissible and non-stable.
\item[(v$^{'}$)] The covers $C_{5}=\{m_{1},m_{2},m_{5}\}$, $C_{9}=\{m_{2},m_{3},m_{4}\}$, $C_{13}=\{m_{2},m_{4},m_{5}\}$ are non-preserved, where
$C_{5^{*}}=\{m_{2}, m_{5}\}$, $m_{1}$, $C_{9^{*}}=\{m_{3}, m_{4}\}$ and $C_{13^{*}}=\{m_{2}, m_{5}\}$  are broken subsets of $C_{5}$, $C_{9}$ and $C_{13}$, respectively.
\end{itemize}
\end{example}

\begin{example}
Consider the following monomial ideal $$I=\langle\underbrace{x_1x_2x_3}_{m_1},\underbrace{x_1x_2x_4}_{m_2},\underbrace{x_1x_3x_4}_{m_3},\underbrace{x_2x_3x_4}_{m_4},\underbrace{x_1x_2x_5}_{m_5},
\underbrace{x_1x_3x_6}_{m_6},\underbrace{x_1x_4x_7}_{m_7}\rangle,$$
and the monomial order in $G(I)=\{m_{1}<m_{2}<m_{3}<m_{4}<m_{5}<m_{6}<m_{7}\}$.
The $E$-minimal covers of $m_{1}$ are
\begin{itemize}
\item $C_1=\{m_1,m_2,m_3\}$,
\item $C_2=\{m_1,m_2,m_4\}$,
\item $C_3=\{m_1,m_2,m_6\}$,
\item $C_4=\{m_1,m_3,m_4\}$,
\item $C_5=\{m_1,m_3,m_5\}$,
\item $C_6=\{m_1,m_4,m_5\}$,
\item $C_7=\{m_1,m_4,m_6\}$,
\item $C_8=\{m_1,m_4,m_7\}$,
\item $C_9=\{m_1,m_5,m_6\}$,
\end{itemize}
The $E$-minimal covers of $m_{2}$ are
\begin{itemize}
\item $C_{10}=\{m_2,m_1,m_3\}$,
\item $C_{11}=\{m_2,m_1,m_4\}$,
\item $C_{12}=\{m_2,m_1,m_7\}$,
\item $C_{13}=\{m_2,m_3,m_4\}$,
\item $C_{14}=\{m_2,m_3,m_5\}$,
\item $C_{15}=\{m_2,m_4,m_5\}$,
\item $C_{16}=\{m_2,m_4,m_6\}$,
\item $C_{17}=\{m_2,m_4,m_7\}$,
\item $C_{18}=\{m_2,m_5,m_7\}$.
\end{itemize}
The $E$-minimal covers of $m_{3}$ are
\begin{itemize}
\item $C_{19}=\{m_3,m_1,m_2\}$,
\item $C_{20}=\{m_3,m_1,m_4\}$,
\item $C_{21}=\{m_3,m_1,m_7\}$,
\item $C_{22}=\{m_3,m_2,m_4\}$,
\item $C_{23}=\{m_3,m_2,m_6\}$,
\item $C_{24}=\{m_3,m_4,m_5\}$,
\item $C_{25}=\{m_3,m_4,m_6\}$,
\item $C_{26}=\{m_3,m_4,m_7\}$,
\item $C_{27}=\{m_3,m_6,m_7\}$.
\end{itemize}
The $E$-minimal covers of $m_{4}$ are
\begin{itemize}
\item $C_{28}=\{m_4,m_1,m_2\}$,
\item $C_{29}=\{m_4,m_1,m_3\}$,
\item $C_{30}=\{m_4,m_1,m_7\}$,
\item $C_{31}=\{m_4,m_2,m_3\}$,
\item $C_{32}=\{m_4,m_2,m_6\}$,
\item $C_{33}=\{m_4,m_3,m_5\}$.
\item $C_{34}=\{m_4,m_5,m_6,m_7\}$
\end{itemize}
We note that
\begin{itemize}
\item[(i)] $m_{5}$, $m_{6}$ and $m_{7}$  have no covers.

\item[(ii)] $C_1=C_{10}=C_{19}$, $C_2=C_{11}=C_{28}$, $C_4=C_{20}=C_{29}$, $C_8=C_{30}$, $C_{13}=C_{22}=C_{31}$, $C_{16}=C_{32}$ and $C_{24}=C_{33}$.
\item[(iii)] The covers $C_{13}$ and $C_{22}$ are the only preserved covers.
\end{itemize}
The set of $L$-admissible symbols of $I$ with dimension $2$ is
\\
$L_2=\{u(m_1;m_2),u(m_1;m_3),u(m_1;m_4),u(m_1;m_5),u(m_1;m_6),u(m_1;m_7),\\
u(m_2;m_5),u(m_2;m_7),u(m_3;m_6),u(m_3;m_7)\}.$
\\
\\
The set of $L$-admissible symbols of $I$ with dimension $3$ is
\\
$L_3=\{u(m_1;m_2;m_5),u(m_1;m_3;m_6),u(m_1;m_2;m_7), u(m_1;m_3;m_7)\}.$
\\
\\
The set of $L$-inadmissible symbols of $I$ with dimension $2$ is
\\
$L_2^{'}=\{u(m_2;m_3),u(m_2;m_4),u(m_2;m_6),u(m_3;m_4),u(m_3;m_5),u(m_4;m_5),
\\u(m_4;m_6),u(m_4;m_7),u(m_5;m_6),u(m_5;m_7),u(m_6;m_7)\}$.
\\
\\
The set of $L$-inadmissible symbols of $I$ with dimension $3$ is
\\
$L_3^{'}=\{u(m_1;m_2;m_3),u(m_1;m_2;m_4),u(m_1;m_2;m_6),u(m_1;m_3;m_4),u(m_1;m_3;m_5),
\\
u(m_1;m_4;m_5),u(m_1;m_4;m_6),u(m_1;m_4;m_7),u(m_1;m_5;m_6),u(m_1;m_5;m_7),u(m_1;m_6;m_7),
\\
u(m_2;m_3;m_4),u(m_2;m_3;m_5),u(m_2;m_3;m_6),u(m_2;m_3;m_7),u(m_2;m_4;m_5),u(m_2;m_4;m_6),
\\
u(m_2;m_4;m_7),u(m_2;m_5;m_6),u(m_2;m_5;m_7),u(m_2;m_6;m_7),u(m_3;m_4;m_5),u(m_3;m_4;m_6),
\\
u(m_3;m_4;m_7),u(m_3;m_5;m_6),u(m_3;m_5;m_7),u(m_3;m_6;m_7),u(m_4;m_5;m_6),u(m_4;m_5;m_7),
\\
u(m_4;m_6;m_7),u(m_5;m_6;m_7)\}$.
\\
We note that
\begin{itemize}
\item[(iv)] The $L$-admissible symbols $u(m_1;m_2;m_7)$ and $u(m_1;m_3;m_7)$, are the only non-stable $L$-admissible symbols.
\item[(iv$^{'}$)] The covers $C_{13}=\{m_{1}, m_{2}, m_{7}\}$ and $C_{22}=\{m_{1}, m_{3}, m_{7}\}$ are the only preserved covers.
\item[(v)] The symbols $u(m_1;m_2;m_5)$ and $u(m_1;m_3;m_6)$, are the only stable $L$-admissible symbols.
\end{itemize}
\end{example}
By the Theorem~\ref{equivalenciasdeideallyubeznik} the Lyubeznik resolution of $I$ with respect to the monomial order given in $G(I)$ is not a minimal free resolution. Therefore, the ideal $I$ is not a Lyubeznik ideal.

\section{The combinatorial calculation}
Finally, in this section we introduce the combinatorial invariants that will allow us to calculate some algebraic invariants, the projective dimension, the length of Lyubeznik, and the arithmetical rank of a monomial ideal.
\\

A {\it clutter\/} $\mathcal{C}$, with finite vertex set
$V$ is a family of subsets of $V$, called {\it edges\/}, none
of which is included in another. The set of
vertices and edges of $\mathcal{C}$ are denoted by $V(\mathcal{C})$
and $E(\mathcal{C})$ respectively. For example, a simple
graph (no multiple edges or loops) is a clutter. For simplicity we mean by clutter to $E(\mathcal{C})$.
\\
An {\it oriented clutter\/} $\mathcal{C}$, with finite vertex set
$(V,\prec)$ (A set totally ordered) is a family of subsets of $V$, called {\it oriented edges\/}, none
of which is included in another. The set of
vertices and oriented edges of $\mathcal{C}$ are denoted by $(V(\mathcal{C}),\prec)$
and $(E(\mathcal{C}),\prec)$ respectively.
\begin{definition}
For an  oriented clutter $\mathcal{C}$, if each edge is not preserved, then $\mathcal{C}$ is called a \emph{Lyubeznik clutter}.
\end{definition}

Let $(G=\{\mu_{1},\mu_{2},\ldots,\mu_{f}\},\prec)$ be a monomial order, and $I=\langle G(I)\rangle$ a monomial ideal in $R=K[x_1,\ldots,x_n]$, i.e., $G(I)=\{\mu_{1},\mu_{2},\ldots,\mu_{f}\}$ be its minimal set of monomial generators.
\\
Consider the {\it oriented clutter of E-minimal covers} $\mathcal{C}_{E}(I)$, with:
$$\mathcal{C}_{E}(I):=\{A\subseteq G(I) : A \textmd{ is an E-minimal cover } \}.$$
\begin{proposition}\label{IlyubeznikssiClyubeznik}
Let $(G(I),\prec)$ be a monomial order, and let $I=\langle G(I)\rangle$. Then $I$ is a Lyubeznik ideal, if and only if, $\mathcal{C}_{E}(I)$ is a Lyubeznik clutter.
\end{proposition}
\begin{proof}
This follows from Theorem~(\ref{equivalenciasdeideallyubeznik}).
\end{proof}
\begin{theorem}\label{beti=preserved}
Let $I=\langle m_{1}\prec m_{2}\prec\cdots \prec m_{\mu}\rangle$ be a Lyubeznik ideal with $\mathcal{L}\bullet_{(I,\prec)}$ the minimal free resolution. Then
\[\beta_{t\, j}(R/I)=\beta_{t-1\, j}(I)\]
\[\beta_{t-1\, j}(I)= |\{u(i_{1};i_{2};\ldots;i_{t}) \,:\, u \textmd{ is L-admissible}; j=deg(lcm(m_{i_{1}},m_{i_{2}},\ldots,m_{i_{t}}))\}|\]
\[\beta_{t-1\, j}(I)= |\{A=\{m_{1}\prec m_{2}\prec\cdots \prec m_{t}\} \,:\, \textmd{ is preserved }; j=deg(lcm(A))\}|.\]
\end{theorem}
\begin{proof}
This follows from Proposition~(\ref{betti-simbolos}), Theorem~(\ref{clasificasimbolos}) and Theorem~(\ref{clasificaconjuntos}).
\end{proof}

Now, given an ideal $I=\langle G(I)\rangle$ with a total order in its set of generators, $(G(I),\prec)$, we will define the obstructions, to measure the defect that has not to be an ideal of Lyubeznik.
\\
The \emph{obstruction of Lyubeznik} with respect to $\prec$:
$$ObsL(I,\prec):=ObsL(\mathcal{C}_{E}(I),\prec):=\textmd{max}\{|A| : A\in\mathcal{C}_{E}(I); A \textmd{ is preserved } \}.$$
When $\{|A| : A\in\mathcal{C}_{E}(I); A \textmd{ is preserved } \}=\emptyset$, we define that $ObsL(\mathcal{C}_{E}(I),\prec)=0$.
\\
The \emph{total obstruction of Lyubeznik}:
$$TObsL(I):=ObsL(\mathcal{C}_{E}(I)):=\textmd{min}\{ObsL(I,\prec) : \prec \textmd{ is a total order } \}.$$
\begin{theorem}\label{LyubeznikssiTObsL=0}
Let $I$ be a monomial ideal. Then $I$ is a Lyubeznik ideal, if and only if, $TObsL(I)=0$.
\end{theorem}
\begin{proof}
This follows from Theorem~(\ref{equivalenciasdeideallyubeznik}).
\end{proof}
In addition, we define two new combinatorial invariants that we capture the information in the resolution of Lyubeznik. Let $(G(I),\prec)$ be a monomial order, and let $I=\langle G(I)\rangle$. The \emph{preserved size}  with respect to $\prec$ and the \emph{preserved size}, by:
$$ps(I,\prec)=\textmd{max}\{|A| : A\subseteq G(I); A \textmd{ is preserved } \},$$
$$ps(I)=\textmd{min}\{ps(I,\prec) : \prec \textmd{ is a total order } \}.$$
\begin{theorem}\label{L=ps}
Let $I$ be a monomial ideal. Then $L(I,\prec)=ps(I,\prec)$ for each total order. Furthermore, $L(I)=ps(I)$.
\end{theorem}
\begin{proof}
This follows from Theorems~(\ref{clasificasimbolos}) and ~(\ref{clasificaconjuntos}).
\end{proof}
\begin{corollary}\label{projdim=L=ps}
Let $I$ be a monomial ideal. If $I$ is a Lyubeznik ideal or $I$ is an almost Lyubeznik ideal, then $projdim(R/I)=L(I)=ps(I)$.
\end{corollary}
\begin{proof}
This follows from Theorem~(\ref{L=ps}).
\end{proof}
\begin{corollary}\label{projdim=L=ps}
Let $I$ be a monomial ideal. If $projdim(R/I)=ps(I)$, then $I$ is an almost Lyubeznik ideal.
\end{corollary}
\begin{proof}
This follows from Theorem~(\ref{L=ps}).
\end{proof}

Now, we present an application to another algebraic invariant, the arithmetical rank. First we introduce the definitions and results needed for our application.
\\
Let $A$ be Noetherian commutative ring with identity. We say that some elements $r_{1},\ldots,r_{m}$ in $A$
generate an ideal $I$ of $A$ up to radical if
$$\sqrt{I}=\sqrt{(r_{1},\ldots,r_{m})}.$$
The smallest $m$ with this property is called the \emph{arithmetical rank} of $I$, denoted by $ara(I)$. \\

The problem of determining the arithmetic rank was
first studied by P. Schenzel and W. Vogel \cite{Schenzel-Vogel}, T. Schmitt and W. Vogel \cite{Schmitt-Vogel}
and G. Lyubeznik \cite{Lyubeznik-rank}. Determi\-ning the arithmetical rank of an ideal
I, or at least a satisfactory upper bound for it, is, in general, a hard task. This problem is open in general.
Some results in this direction have already been proven in several works. The arithmetical rank of every ideal in the polynomial
ring $R=K[x_{1},...,x_{n}]$ (where $K$ is a field) is at most $n$, \cite{cota-n}.
\\

A better lower bound is provided, in general, by the local cohomological
dimension, which, for any squarefree monomial ideal, coincides with the
projective dimension.
\\
Let $projdim_{R}(R/I)$ the \emph{projective dimension} of $R/I$, i.e., the length of a minimal free
resolution of $R/I$. Let $H_{I}^{i}(R)$ denote the $i-th$ local cohomology mo\-du\-le of $R$ with respect to $I$. The
\emph{cohomological dimension} of $I$ is defined to be the natural number:
$$cd(I)=max\{i | H_{I}^{i}(R)\neq0\}.$$
We shall throughout suppose that $R$ is the polynomial ring $K[x_1,\ldots,x_n]$. From the expression of the local cohomology modules in terms of $\check{C}$ech
complex, one can easily see that (\cite[p.~414, Example~2]{Hartshorne-Coho-dim}, \cite[Theorem~5.4]{Huneke-desigualdad}) for
all ideals $I$ in a commutative Noetherian ring $$cd(I)\leq ara(I).$$
We recall that for $I$ monomial ideal, $ara(I)=ara(\sqrt{I})$ with $\sqrt{I}$ a squarefree monomial ideal  (See \cite{Barile-2005}). By Lyubeznik \cite[Theorem~1]{Lyubeznik-local-coho}, for all squarefree monomial ideal $I$ one has that \\
 $projdim(R/I)=cd(I)$. Therefore
\begin{equation*}\label{ht-projdim-cd-ara-m}
ht(I)\leq projdim(R/I)=cd(I)\leq ara(I)\leq\mu(I).
\end{equation*}
Where $ht(I)$ is the \emph{height} and $\mu(I)$ is the minimum number of generators.
In particular, if $I$ is a set-theoretic complete intersection, then $R/I$ is Cohen–
Macaulay.
\begin{corollary}\label{ara=projdim=L=ps}
Let $I$ be a squarefree monomial ideal. Then $ara(I)\leq L(I)=ps(I)$. In particular, if $I$ is a Lyubeznik ideal or $I$ is an almost Lyubeznik ideal ($L(I)=projdim(R/I)$), then
$$L(I)=ps(I)=projdim(R/I)=projdim(I) +1\leq ara(I).$$
\end{corollary}

The most important theorem that relates to the Lyubeznik resolution and the arithmetical rank is as follows:

\begin{theorem}\label{kimura-llongitud1}\cite{Kimura-resolucionlyubeznik}
Let $I$ be a monomial ideal of $R$. Then $$ara(I)\leq L(I).$$
\end{theorem}
\begin{corollary}\label{kimura-llongitud2}
Let $I$ be a squarefree monomial ideal of $R$. If $I$ is a Lyubeznik ideal or $I$ is an almost Lyubeznik ideal ($L(I)=projdim(R/I)$), then $$ara(I)=projdim(R/I)=L(I)=ps(I).$$
\end{corollary}

Making the duality Combinatorics-algebra with the algebraic contribution of Kimura \cite{Kimura-resolucionlyubeznik}, we provide the $L(I)$ generators of $\sqrt{I}$ (unless radical), obtained from combinatorial way. Let $I=\langle m_{1}\prec m_{2}\prec\cdots \prec m_{\mu}\rangle$ be a Lyubeznik ideal with $\mathcal{L}\bullet_{(I,\prec)}$ the minimal free resolution and $\Delta_{(I,\prec)}$ the Lyubeznik complex
of $I$ under the total order $\prec$, i.e.,
$$\Delta_{(I,\prec)}=\{F\in G(I) : min\{u\in G(I): u|lcm(H)\}\in H \textmd{  for all  } H\subseteq F\},$$
$$\Delta_{(I,\prec)}=\{F\in G(I) : min\{u\in G(I): F \textmd{  is preserved  }\},$$
with $dim(\Delta_{(I,\prec)})=max\{|F|-1 : F\in \Delta_{(I,\prec)}\}=L(I,\prec)-1$.
\\
The $\lambda:=L(I,\prec)$ elements $g_{1}, g_{2}, \ldots , g_{\lambda}$ such that $\sqrt{\langle g_{1}, g_{2}, \ldots , g_{\lambda}\rangle}=\sqrt{I}$, are obtained as follows:
\\
$g_{1}=m_{1},$
\\
$g_{2}=m_{2}+\sum_{\{m_{i_{1}\geq3}\prec m_{i_{2}}\prec\cdots\prec m_{i_{\lambda-1}}\}\in\Delta_{\lambda-1}}m_{i_{1}}m_{i_{2}}\cdots m_{i_{\lambda-1}},$
\\
$\vdots$
\\
$g_{s}=m_{s}+\sum_{\{m_{i_{1}\geq s+1}\prec m_{i_{2}}\prec\cdots\prec m_{i_{\lambda-s+1}}\}\in\Delta_{\lambda-s+1}}m_{i_{1}}m_{i_{2}}\cdots m_{i_{\lambda-s+1}},$
\\
$\vdots$
\\
$g_{\lambda}=m_{\lambda}+\sum_{\{m_{i_{1}\geq \lambda+1}\}\in\Delta_{1}}m_{i_{1}} = m_{\lambda}+m_{\lambda+1}+\cdots+m_{\mu}.$
Where, $\Delta_{i}=\{F\in \Delta_{(I,\prec)} : |F|=i\}$.
\\

Finally, we consider the case of a \textit{simple graph} $\mathcal{G}$ (no multiple edges or loops), with finite vertex set
$V=\{x_{1},...,x_{n}\}$. The set of
vertices and edges of $\mathcal{G}$ are denoted by $V(\mathcal{G})$
and $E(\mathcal{G})$ respectively. The {\it edge ideal\/} of $\mathcal{G}$,
denoted by $I(\mathcal{G})$, is the ideal of $R$
generated by all monomials $m_e=x_{i}x_{j}$ such
that $e=\{x_{i},x_{j}\}\in E(\mathcal{G})$. The map
$$\mathcal{G}\longmapsto I(\mathcal{G})$$
gives an one-to-one
correspondence between the family of simple graphs and the family of
squarefree monomial ideals. Edge ideals of graphs were introduced and studied by Villarreal in \cite{Vi2}.
\begin{definition}
A Lyubeznik ideal is a \emph{Lyubeznik totally ideal} if for any
total order on $G(I)$ the corresponding Lyubeznik resolution of $R/I$ is a minimal free resolution of
$R/I$.
\end{definition}
\begin{proposition}\label{graficasTLI}
Let $\mathcal{G}$ a simple graph and $I(\mathcal{G})$ the edge ideal of $\mathcal{G}$. If $\mathcal{G}$ if do not have paths of length $\geq3$, then
$I(\mathcal{G})$ is a Lyubeznik totally ideal.
\end{proposition}
\begin{proof}
Note that each edge $\{x_{i}x_{j}\}$ do not have $E$--minimal covers.
\end{proof}
\begin{proposition}\label{graficasTLI}
Let $\mathcal{G}$ a simple graph, $G=\{ab,bc,ac\}$ a cycle of length $3$, and $I(\mathcal{G})$ the edge ideal of $\mathcal{G}$. Then
$I(\mathcal{G})$ is a Lyubeznik totally ideal.
\end{proposition}
\begin{proof}
Note that the only $E$--minimal cover is $C=\{ab,bc,ac\}$ and is not preserved for any total order.
\end{proof}
\begin{proposition}\label{graficasTLI}
Let $\mathcal{G}$ a simple graph and $I(\mathcal{G})$ the edge ideal of $\mathcal{G}$. If $\mathcal{G}$ do not have paths of length $\geq4$, then
$I(\mathcal{G})$ is a Lyubeznik ideal.
\end{proposition}
\begin{proof}
Note that $\mathcal{G}$ do not have paths of length $\geq4$, then the $E$--minimal covers are disjoint. Furthermore, each $E$--minimal cover of  $\{x_{i}x_{j}\}$,(if exist), is the cycle $\{x_{r}x_{i},x_{i}x_{j},x_{j}x_{r}\}$ or has the form $\{x_{r}x_{i},x_{i}x_{j},x_{j}x_{s}\}$, in the first case is not preserved under any order, and in the second case is not preserved considering the total order $x_{i}x_{j}\prec x_{r}x_{i}\prec x_{j}x_{s}$.
\end{proof}
\begin{proposition}\label{graficasTLI-2}
Let $\mathcal{G}$ a simple graph, $G=\{ab,bc,cd,ad\}$ a cycle of length $4$, and $I(\mathcal{G})$ the edge ideal of $\mathcal{G}$. Then
$I(\mathcal{G})$ is not a Lyubeznik ideal.
\end{proposition}
\begin{proof}
With any order, (for example $ab\prec bc\prec cd\prec ad$), we have a cover that is preserved, (in the example is $\{ab,bc,cd\}$).
\end{proof}
\begin{definition}
Let $I=\langle G(I)\rangle$ be a monomial ideal with $(G(I)=\{m_{1}, m_{2}, \cdots, m_{\mu}\},\prec)$ a total order. We say that $u$ is a \emph{possible court} of $D\subseteq G(I)$
if $u|lcm(D)$ and $u\notin D$. Let us remember that if in addition $u\prec D$, then $u$ is a \emph{court}.
\end{definition}
\begin{proposition}\label{posiblescortesorder}
Let $I=\langle G(I)\rangle$ be a monomial ideal with $(G(I)=\{m_{1}, m_{2}, \cdots, m_{\mu}\})$. Let $P=\{m_{i}\}$ be the set of the possible courts. Then any total order $(G(I)=\{m_{1}, m_{2}, \cdots, m_{\mu}\},\prec)$ such that $m_{i}\prec m$ for all $m\in G(I)\setminus P$ satisfied that $I$ is a Lyubeznik ideal witn $\prec$.
\end{proposition}
\begin{proof}
For every element $u$ of $G(I)$ and every
cover $C$ of $u$, $C$ is not preserved with court $m_{i}$.
\end{proof}


\begin{thebibliography}{XX}

\bibitem{Barile-2005} { \sc M. Barile}, { \em On ideals whose radical is a monomial ideal}, Comm. Algebra \textbf{33} (2005), 4479–-4490.

\bibitem{cota-n} {\sc D. Eisenbud and E. G. Evans Jr}, {\em Every algebraic set in $n$--space is the intersection of $n$ hypersurfaces.}
Invent. Math. \textbf{19} (1973), 107–-112.

\bibitem{Guo-Wu-Yu} {\sc J.~Guo, T.~Wu and H.~Yu,} \textit{On monomial ideals whose Lyubeznik resolution is minimal}  arXiv:1312.0321, (2013).

\bibitem{Hartshorne-Coho-dim} {\sc R. Hartshorne}, {\em Cohomological dimension of algebraic varieties}, Ann. of
Math. \textbf{88} (1968), 403--450.

\bibitem{Huneke-desigualdad} {\sc C. Huneke}, {\em Lectures on local cohomology} (with an appendix by Amelia Taylor) (2004).

\bibitem{Kimura-resolucionlyubeznik} {\sc K. Kimura},
{\em Lyubeznik resolutions and the arithmetical rank of monomial ideals.} (English summary)
Proc. Amer. Math. Soc.  \textbf{137}  (2009),  no. 11, 3627–-3635.

\bibitem{Lyubeznik-local-coho} {\sc G.~Lyubeznik,} \textit{On the local cohomology modules $H_{\mathcal{A}}^{i}(R)$ for ideals $\mathcal{A}$ generated by monomials in an R-sequence.} (``Complete intersections", Lecture Notes in Mathematics Vol. 1092) Berlin Heidelberg New York Tokyo: Springer (1984).

\bibitem{Lyubeznik-resolution} {\sc G.~Lyubeznik,} \textit{A new explicit finite free resolution of ideals generated by monomials in an
R-sequence.} J. Pure Appl. Algebra \textbf{51} (1988), 193--195.


\bibitem{Lyubeznik-rank} {\sc G.~Lyubeznik,} \textit{On the arithmetical rank of monomial ideals.} J. Algebra \textbf{112} (1988), 86--89.

\bibitem{Novik} {\sc I.~Novik,} \textit{Lyubesnik's Resolution and Rooted Complexes.} Journal of Algebraic Combinatorics \textbf{16} (2002), 97--101.

\bibitem{Schenzel-Vogel} {\sc P. Schenzel and W. Vogel}, {\em On set-theoretic intersections}, J. Algebra \textbf{48} (1977),
no. 2, 401-–408.

\bibitem{Schmitt-Vogel} {\sc T. Schmitt and W. Vogel}, {\em Note on set-theoretic intersections of subvarieties of projective space}, Math. Ann. \textbf{245} (1979), no. 3, 247-–253.

\bibitem{Vi2} {\sc R. H. Villarreal}, {\em Cohen-{M}acaulay graphs}, Manuscripta
Math. {\bf 66} (1990), 277--293.

\end{thebibliography}
\end{document}